\documentclass[a4paper,12pt]{scrartcl}
\usepackage{subfigure,german,fancyhdr, pstricks}
\usepackage{amssymb,amsmath,amsthm,inputenc}
\usepackage{bbm}
\usepackage{multido}
\usepackage{oldgerm}
\usepackage[all]{xy}
\let\SAVEDRightarrow=\Rightarrow
\usepackage{marvosym}
\let\Rightarrow=\SAVEDRightarrow
\usepackage{pifont}
\usepackage{array}
\usepackage{graphicx}

\theoremstyle{plain}

\newtheorem{satz}{Satz}[section]

\newtheorem{lemma}[satz]{Lemma}

\newtheorem{prop}[satz]{Proposition}
\newtheorem{thm}[satz]{Theorem}
\newtheorem{cor}[satz]{Corollary}

\theoremstyle{definition}
\newtheorem{defi}[satz]{Definition}

\newtheorem{rmrk}[satz]{Remark}

\newtheorem{example}[satz]{Example}

\theoremstyle{remark}

\DeclareMathOperator{\st}{st}

\DeclareMathOperator{\Cone}{Cone}
\DeclareMathOperator{\diam}{diam}
\DeclareMathOperator{\Isom}{Isom}

\newcommand{\mc}[1]{\mathcal{#1}}

\newcommand{\up}[1]{\!\! ~^{#1}}

\renewcommand{\bar}{\overline}
\renewcommand{\phi}{\varphi}
\renewcommand{\theta}{\vartheta}
\renewcommand{\epsilon}{\varepsilon}
\renewcommand{\rho}{\varrho}
\newcommand{\IN}{\mathbbm{N}}

\newcommand{\IR}{\mathbbm{R}}

\numberwithin{equation}{section}

\DeclareFontFamily{U}{schwell}{}
\DeclareFontShape{U}{schwell}{m}{n}{
   <8> <9> <10> <10.95> <12> <14.4> <17.28>  <20.74> <24.88> schwell}{}
\DeclareMathAlphabet{\schwell}{U}{schwell}{m}{n}
\newcommand\textschwell{\usefont{U}{schwell}{m}{n}}
\DeclareTextFontCommand{\schwell}{\textschwell}
\DeclareFontFamily{U}{suet}{}
\DeclareFontShape{U}{suet}{m}{n}{
   <8> <9> <10> <10.95> <12> <14.4> <17.28>  <20.74> <24.88> suet14}{}
\DeclareMathAlphabet{\suet}{U}{suet}{m}{n}
\newcommand\textsuet{\usefont{U}{suet}{m}{n}}
\DeclareTextFontCommand{\suet}{\textsuet}
\DeclareMathAlphabet{\dis}{T1}{cmss}{bx}{sl}
\input{cyracc.def}
\newfont{\cyrfnt}{wncyr10}

\newfont{\cybfnt}{wncyb10}

\newfont{\cyifnt}{wncyi10}

\newfont{\cyscfnt}{wncysc10}

\newfont{\cyssfnt}{wncyss10}

\title{\LARGE Slow ultrafilters and asymptotic cones of proper metric spaces}
\author{\normalsize Lars Scheele\footnote{E-Mail address: lars.scheele@uni-muenster.de}\\{\it \normalsize Universit\"at M\"unster, Einsteinstr. 60, 48149 M\"unster, Germany}}
\date{\normalsize October 8, 2010}

\begin{document}
\maketitle
\begin{abstract}
In this paper I present an elementary construction to prove that any proper metric space can arise as the asymptotic cone of another proper metric space. Furthermore I answer a question of Dru\c{t}u and Sapir concerning slow ultrafilters.
\end{abstract}

\section{Introduction}

Given a metric space $X$, the asymptotic cone of $X$ is another metric space meant to capture the ``large-scale geometry'' of $X$. This construction has been introduced by Gromov and was later modified by van den Dries and Wilkie who used non-standard methods to ensure that the asymptotic cone exists for every metric space. However one of the drawbacks is that it is a priori not clear in how far the cone of $X$ depends on additional data one has to choose, namely a non-principal ultrafilter and a sequence of scaling factors. \\
In this short paper I would like to briefly recall some basic definitions about ultrafilters and ultraproducts in Section 2 and then define asymptotic cones in Section 3. The question by Gromov whether there is an example of a finitely generated group with two different (i.e. non-homeomorphic) asymptotic cones has been answered positively by Thomas and Velickovic. This example will also be discussed in that section.\\
In Section 4 I then show how much freedom one has in the construction of the cone. One of the two main results (Theorem \ref{decone}) states that any proper metric space can be realized as an asymptotic cone of another metric space, which again will be proper, so that the process can be iterated. In a very recent preprint (\cite{S}) Sisto independently obtains the same result using a very different construction and non-standard methods. He also proves that any seperable metric space which is an asymptotic cone has to be proper.\\
The known examples of spaces with different cones all depend on the choice of very fast growing sequences. Dru\c{t}u and Sapir suggested a way to avoid these fast sequences and asked if there are still examples of groups with different cones with respect to these slow ultrafilters. The second main result of this paper (Theorem \ref{Drutuanswer}) answers this question in showing that any cone that can be realised with a fast growing sequence can also be obtained using a slow ultrafilter.\\[1ex]
All of the results in this paper will appear in my doctoral thesis. The author is indebted to his thesis advisor Katrin Tent for helpful questions and even more helpful answers and discussions.

\section{Ultrafilters and ultraproducts}

\begin{defi}
Let $I$ be a set. A filter $\mu$ on $I$ is a nonempty collection of subsets of $I$, such that for all subsets $A,B \subseteq I$ we have
\begin{itemize}
\item[i)] $\emptyset \notin \mu$.
\item[ii)] $A \in \mu, A \subseteq B \Rightarrow B \in \mu$.
\item[iii)] $A,B \in \mu \Rightarrow A \cap B \in \mu$.
\end{itemize}
The set of all filters on $I$ can be partially ordered by inclusion. It is easy to see that totally ordered subsets have upper bounds and therefore maximal filters exist by Zorn's lemma. Those are called ultrafilters. They can be characterized as follows: A filter $\mu$ is an ultrafilter if and only if
\begin{itemize}
\item[iv)] For all $A \subseteq I$ either $A \in \mu$ or $I \backslash A \in \mu$.
\end{itemize}
\end{defi}

An ultrafilter on $I$ can also be regarded as a finitely additive probability measure on $I$, which only takes the values 0 and 1. We say that some property of elements of $I$ holds $\mu$-almost everywhere ($\mu$-a.e.) if the set where it holds lies in $\mu$.

\begin{example}
Let $I$ be a set and $i \in I$ a point. Then the collection
\[ \mu := \{ A \subseteq I : i \in A \} \]
defines an ultrafilter on $I$. Such an ultrafilter is called principal.
\end{example}

Note that for finite sets $I$ each ultrafilter is of this form. Non-principal ultrafilters on $I$ exist if and only if $I$ is infinite: Take the collection of all cofinite sets in an infinite $I$. This is a filter and therefore contained in an ultrafilter, which is non-prinicipal since it contains no finite sets.

\begin{defi}
Let $X$ and $I$ be sets and $\mu$ a non-prinicipal ultrafilter on $I$. The ultraproduct of $X$ with respect to $\mu$ and $I$ is defined as
\[ \up{*}X := \prod_\mu X := \prod_I X / \sim \]
where $\sim$ is the equivalence relation on the product given by
\[ (x_i) \sim (y_i) : \iff x_i = y_i \; \mu\mbox{-a.e.} \iff \{ i \in I : x_i = y_i \} \in \mu \quad \mbox{for } (x_i),(y_i) \in \prod_I X.\]
An equivalence class modulo $\sim$ will be denoted by $[x_n]$.
\end{defi}

The important thing to note here is that if $X$ carries additional structure, this can be carried over to the ultraproduct $\up{*}X$. \L o\v{s}'s theorem (cf. \cite{BS}) states that any first-order sentence true in $X$ can be transferred to $\up{*}X$. As an example consider the field of hyperreal numbers.

\begin{example}
Fix a non-prinicipal ultrafilter $\mu$ on $\IN$ and consider
\[ \up{*}\IR := \prod_\mu \IR. \]
These are just sequences $(x_n)$ of real numbers, where two sequences are identified if they agree $\mu$-a.e. Then $\up{*}\IR$ carries the structure of an ordered field. It is clear that the set of sequences forms an integral domain if addition and multiplication os defined elementwise. After the identification $\sim$ one really obtains a field: For any sequence $(x_n)$ of real numbers either the set $\{ n \in \IN : x_n = 0 \}$ is in $\mu$ or its complement. In the first case, $[x_n] = 0$ in $\up{*}\IR$ and in the second one can define
\[ y_n := \left\{ \begin{array}{ll} x_n^{-1} \quad & \mbox{if } x_n \not= 0 \\ 0 \quad & \mbox{if } x_n = 0\end{array}\right. \]
Then $(x_n \cdot y_n)$ is $\mu$-a.e. equal to 1, so $[x_n]^{-1} = [y_n]$.\\[1ex]
The order $<$ can also be transferred to make $\up{*}\IR$ into an ordered field, which is real closed but not archimedean.
\end{example}

The field $\IR$ can be embedded into $\up{*}\IR$ by taking constant sequences. An element $x \in \up{*}\IR$ is called finite if there is some constant $C > 0$ such that $|x| < C$, otherise it is called infinite. Further $x$ is called infinitesimal if $x \not= 0$ but $|x| < \epsilon$ for all $\epsilon > 0, \epsilon \in \IR$.\\[1ex]
The set of all finite elements of $\up{*}\IR$ forms a local ring with the set of all infinitesimal elements as maximal ideal. The quotient is isomorphic to $\IR$ and the projection map $\st$ is called the ``standard part''.\\[1ex]
Another way of looking at this is the following. Any finite element $x = [x_n] \in \up{*}\IR$ corresponds to a bounded sequence $(x_n)$ which has a unique limit with respect to $\mu$, i.e. a number $a \in \IR$ such that every neighbourhood of $a$ contains $\mu$-almost every element of the sequence $(x_n)$. Write
\[ \lim_{n,\mu} x_n = a \qquad \mbox{or simply} \quad \lim_\mu x_n = a \]
for this limit. Note that it depends a lot on the choice of $\mu$. The bounded sequence $(-1)^n$ has $\mu$-limit 1 or $-1$ depending on whether the set of even or the set of odd natural numbers lies in $\mu$.\\[1ex]
Taking the limit of a bounded sequence is the same as taking the standard part of a finite hyperreal number:
\[ \st\big([x_n]\big) = \lim_\mu x_n \qquad \mbox{if } [x_n] \in \up{*}\IR \mbox{ is finite.}\]

\section{Asymptotic Cones}
We start by defining the asymptotic cone of an arbitrary (pseudo)-metric space and discuss to what extend it depends on the defining data. The ideas here are not new and can be found in \cite{R} and \cite{DS}. Recall that in a pseudo-metric space all the axioms for metric spaces are valid except for the possibility that two different points can have distance 0.
\begin{defi}
Let $(X,d)$ be a pseudo-metric space, $\mu$ a non-principal ultrafilter on $\IN$, $(e_n)$ a sequence of points in $X$ (the ``sequence of base-points'') and $(\alpha_n)$ a sequence of positive real numbers tending to infinity (the ``sequence of scaling factors''). Consider the ultrapower $\up{*}X$, which is an $\up{*}\IR$-pseudo-metric space. This $\up{*}\IR$ metric will be denoted by $\up{*}d$.\\[1ex]
Set $e := [e_n] \in \up{*}X$ and $\alpha := [\alpha_n] \in \up{*}\IR$. The metric $\up{*}d/\alpha$ is again a $\up{*}\IR$ metric on $\up{*}X$. Consider now the following set:
\[ \up{*}X^\alpha_e := \left\{ [x_n] \in \up{*}X : \frac{\up{*}d\big([x_n],[e_n]\big)}{\alpha} \mbox{ is finite in } \up{*}\IR \right\}.\]
For any finite non-standard real number one can take the standard part, which is real. This makes $\up{*}X^\alpha_e$ into a pseudo-metric space. Identifying points with distance 0 gives the asymptotic cone of $X$:
\[ \Cone_\mu(X,e,\alpha) := \up{*}X^\alpha_e / \approx \quad \mbox{, where } [x_n] \approx [y_n] \iff \frac{\up{*}d\big([x_n],[y_n]\big)}{\alpha} \mbox{ is infinitesimal.}\]
We don't want to complicate the notation even more, so we will denote an equivalence class with respect to $\approx$ again by $[x_n]$. The metric $d_\infty$ on $\Cone_\mu(X,e,\alpha)$ is defined by
\[ d_\infty \big([x_n],[y_n]\big) := \st \left( \frac{\up{*}d \big( [x_n], [y_n] \big)}{\alpha} \right) = \lim_\mu \frac{d(x_n,y_n)}{\alpha_n}.\]
\end{defi}
\begin{rmrk}
Saturation properties of ultrapowers guarantee that the asymptotic cone is always a complete metric space\footnote{An ultraproduct over a countable set is always $\aleph_1$-saturated, cf. \cite{M}, Exercise 4.5.37. A limit of a Cauchy-sequence can be written as the realization of a type over a countable set and from this the assertion follows directly. Note that the asymptotic cone itself will not be saturated in general.}. A more direct proof can for example be found in \cite{VDW}, Proposition 4.2.
\end{rmrk}
\begin{defi}
For later use we also define iterated asymptotic cones. For this note that if $X$ is a metric space and $e \in \up{*}X$ a fixed basepoint, the asymptotic cone of $X$ will have a canonical basepoint given by the equivalence class of $e$, which we will denote by $\hat{e}$. Fix a non-principal ultrafilter $\mu$ on $\IN$ and an infinite hyperreal number $\alpha$. Then we set $\Cone^0_\mu(X,e,\alpha) := X$ and for $i \in \IN$ set
\[ \Cone^{i+1}_\mu(X,e,\alpha) := \Cone_\mu\big(\Cone^i_\mu(X,e,\alpha),\hat{e},\alpha \big).\] 
\end{defi}
As indicated in the notation, the definition of the asymptotic cone depends on the choices of the ultrafilter $\mu$, the sequence of base points $e$ and the sequence of scaling factors $\alpha$. We want to discuss how severe these dependencies are. The first, alsmost obvious observation is the following.
\begin{lemma}
\label{bounded_add}
Let $\mu$ be a non-principal ultrafilter, $\alpha \in \up{*}\IR$ an infinite hyperreal number and let $\beta \in \up{*}\IR$ be a finite number. Let further $X$ be a metric space with basepoint $e \in X$. Then
\[ \Cone_\mu(X,e,\alpha) \cong \Cone_\mu(X,e,\alpha + \beta).\]
\end{lemma}
\begin{proof}
This is the following simple fact about real sequences. If $x_n$ is any sequence of real numbers, $\alpha_n$ another sequence tending to infinity and $\beta_n$ a $\mu$-a.s. bounded sequence, such that $\frac{x_n}{\alpha_n}$ converges with respect to $\mu$. Then
\[ \lim_\mu \frac{x_n}{\alpha_n + \beta_n} = \lim_\mu \frac{x_n}{\alpha_n}.\]
\end{proof}
\begin{defi}
A metric space $(X,d)$ is called quasi-homogenous if the action of $\Isom X$ has a bounded fundamental domain in $X$. Put another way, $(X,d)$ is quasi-homogenous if $\diam\big(X / \Isom X) < \infty$. Recall that a metric space is called homogenous if the isometry groups acts transitively on the points of $X$.
\end{defi}
\begin{lemma}
Let $(X,d)$ be a metric space, $\mu$ a non-principal ultrafilter on $\IN$ and $\alpha \in \up{*}\IR$ a sequence of scaling factors as above. Let $e$ and $e'$ be two basepoints in $\up{*}X$. If $(X,d)$ is quasi-homogenous, there exists an isometry 
\[ \phi : \Cone_\mu(X,e,\alpha) \to \Cone_\mu(X,e',\alpha) \]
mapping $e$ to $e'$.
\end{lemma}
\begin{proof}
By assumption there is a constant $C > 0$ and isometries $\phi_n \in \Isom X$ such that
\[ d\big(\phi_n(e_n), e_n'\big) < C.\]
This induces a well-defined map $\phi: \Cone_\mu(X,e,\alpha) \to \Cone_\mu(X,e',\alpha)$, which can be seen as follows. Let $x = [x_n] \in \up{*}X^\alpha_e$, then
\begin{eqnarray*}
\frac{\up{*}d\big(\phi(x),e'\big)}{\alpha} = \frac{\up{*}d\big([\phi_n(x_n)],[e_n']\big)}{\alpha} &\leq& \frac{\up{*}d\big([\phi_n(x_n)],[\phi_n(e_n)]\big)}{\alpha} + \frac{\up{*}d\big([\phi_n(e_n)],[e_n']\big)}{\alpha}\\
&\leq& \underbrace{\frac{\up{*}d(x,e)}{\alpha}}_{\mbox{finite}} + \frac{C}{\alpha}.
\end{eqnarray*}
This shows that $\phi(x) \in \up{*}X^\alpha_{e'}$ and therefore the map $\phi$ is well-defined. It is clear that it is an isometry. The calculation above also shows $\phi(e) = e'$, because $C/\alpha$ is infinitesimal and therefore $d_\infty\big(\phi(e),e'\big) = 0$.
\end{proof}
\begin{rmrk}
This proof also shows that for two basepoints $e,e' \in \up{*}X$ with finite distance (in $\up{*}\IR$) the identity map is an isometry. This is especially the case if $e$ and $e'$ are constant, i.e. points of $X$.
\end{rmrk}
Assume from now on unless stated otherwise the basepoint to be one point $e \in X$ embedded into $\up{*}X$ via a constant sequence.\\[1ex]
The dependence on the ultrafilter $\mu$ and the scaling factor $\alpha$ is more crucial. There is an example of a metric space $(X,d)$ having non-homeomorphic cones $\Cone_\mu(X,e,\alpha)$ and $\Cone_{\mu'}(X,e,\alpha)$, where $\mu$ and $\mu'$ are distinct ultrafilters on $\IN$, see \cite{TV}. The construction in this paper can be adapted to give an example of non-homeomorphic cones $\Cone_\mu(X,e,\alpha)$ and $\Cone_\mu(X,e,\beta)$ for different scaling factors $\alpha$ and $\beta$. Indeed, the choices of the ultrafilter and the sequence of scaling factors are interrelated. We will discuss the example in \cite{TV} in greater detail in \ref{TV_example}.

\begin{defi}
Let $\alpha_n$ be a sequence of positive real numbers tending to infinity. We say that this sequence has bounded accumulation if there is a number $N \in \IN$, such that for all $r \in \IN$ the set
\[ S_r = \{ n \in \IN : \alpha_n \in [r,r+1[ \} = \{ n \in \IN : \lfloor \alpha_n \rfloor = r \}\]
has less than $N$ elements.\\[1ex]
If $\mu$ is any ultrafilter on $\IN$, we say that $\alpha$ has $\mu$-almost surely bounded accumulation if there is a set $T \in \mu$, such that $|T \cap S_r|$ is uniformly bounded.\\[1ex]
\end{defi}

\begin{rmrk}
Since $\IN \in \mu$ for all ultrafilters $\mu$ we have that if $\alpha$ has bounded accumulation it also has $\mu$-almost surely bounded accumulation for all $\mu$. Moreover if $\alpha$ has $\mu$-almost surely bounded accumulation for some ultrafilter $\mu$ then there exists a set $A' \in \mu$ such that for each $r \in \IN$ the set
\[ A' \cap S_r = \{ n \in A' : \alpha_n \in [r,r+1[ \} = \{ n \in A' : \lfloor \alpha_n \rfloor = r \} \]
has at most one element. Indeed by assumption exists an $N \in \IN$ and a set $T \in \mu$ such that for all $r \in \IN$ we have $|T \cap S_r| \leq N$. Therefore we can write $T$ as a finite disjoint union
\[ T = A_1 \dot\cup A_2 \dot\cup \dots \dot\cup A_N \]
with $|A_i \cap S_r| \leq 1$ for each $i \leq N$ and each $r \in \IN$. Because the union is disjoint exactly one of the $A_i$ lies in $\mu$ and this can be taken as $A'$.
\end{rmrk}

In the following denote the hyperreal number $[n]$ by $\omega$. This is often used as the ``standard'' scaling sequence.

\begin{prop}[\cite{R}, Appendix B]
\label{scaling}
Let $(X,d)$ be a metric space, $\mu$ a non-principal ultrafilter on $\IN$, $e$ a basepoint and $\alpha$ a sequence of scaling factors. Then there exists a non-principal ultrafilter $\mu'$ on $\IN$, a basepoint $e'$ and an isometric embedding $\phi : \Cone_{\mu'}(X,e',\omega) \to \Cone_\mu(X,e,\alpha)$.\\[0,5ex]
If moreover $\alpha$ has $\mu$-almost surely bounded accumulation, then $\phi$ is an isometry.
\end{prop}

\begin{proof}
Define a map $\psi: \IN \to \IN$ by setting $\psi(n) := \lfloor \alpha_n \rfloor$. Indeed it is no loss of generality to assume $\alpha_n \in \IN$ for all $n$, since $\Cone_\mu(X,e,\alpha)$ and $\Cone_\mu(X,e,\lfloor \alpha \rfloor)$ are isometric by Lemma \ref{bounded_add}.\\[1ex]
Define the ultrafilter $\mu'$ as follows. For any subset $A \subseteq \IN$ set
\[ A \in \mu' : \iff \psi^{-1}(A) \in \mu.\]
It is clear that this defines a non-principal ultrafilter on $\IN$. For every $[x_n] \in \Cone_{\mu'}(X,e,\omega)$ set $\phi\big([x_n]\big) := [x_{\psi(n)}]$. This is a well defined map to $\Cone_\mu(X,e,\alpha)$. Let $[x_n] \in \up{*}X^\omega_e$ be a representative of any point in $\Cone_{\mu'}(X,\omega,e)$ and consider $\phi(x)$:
\[ \frac{\up{*}d\big(\phi([x_n]),e\big)}{[\alpha_n]} = \frac{\up{*}d\big(x_{\psi(n)},e\big)}{[\psi(n)]}\]
and this is a finite hyperreal with respect to $\mu$, because
\[ \frac{\up{*}d\big([x_n],e\big)}{\omega} = \frac{\up{*}d\big([x_n],e\big)}{[n]} \]
is by assumption a finite hyperreal with respect to $\mu'$. This shows $\phi(x) \in \up{*}X^\alpha_e$. We also have for any two representatives $x = [x_n]$ and $y = [y_n]$ of points in $\Cone_{\mu'}(X,\omega,e)$:
\[ \frac{\up{*}d\big(\phi([x_n]),\phi([y_n])\big)}{[\alpha_n]} = \frac{\up{*}d\big([x_{\psi(n)}],[y_{\psi(n)}]\big)}{[\psi(n)]}\]
The $\mu$-limit of this number is the same as the $\mu'$-limit of $\up{*}d(x,y)/\omega$ and this shows that the map $\phi$ respects the distance and is therefore an isometric embedding. Since the asymptotic cone is a metric (not a pesudo-metric) space, it follows in particular that $\phi$ is injective.\\[1ex]
Assume now that $\alpha$ has $\mu$-almost surely bounded accumulation and consider again for each $r \in \IN$ the set
\[ S_r = \{ n \in \IN : \lfloor \alpha_n \rfloor = r \} = \psi^{-1}\big(\{r\}\big).\]
By assumption there exists a set $A \subseteq \IN$ with $A \in \mu$ and $|A \cap S_r| \leq 1$ for each $r \in \IN$.\\[1ex]
Consider then $A' := \psi(A)$. By construction we have $\psi^{-1}(A') = A$ and therefore $A' \in \mu'$. The inverse of $\phi$ can then be defined on the set of indices in $A'$ and it follows that $\phi$ is surjective and therefore an isometry.
\end{proof}

We will now see what this proof shows in the particular example of Thomas and Velickovic.

\begin{example}
\label{TV_example}
In the paper \cite{TV} the authors give an example of a metric space $(X,d)$ (in this case a finitely generated group $G$ with the word metric) and two different asymptotic cones with respect to two distinct ultrafilters. In particular they prove the following:\\[1ex]
There is a metric space $(X,d)$ and two disjoint subsets $A,B \subseteq \IN$ such that for any ultrafilter $\mu$ containing $A$ the cone $\Cone_\mu(X,e,\omega)$ is simply connected whereas for any ultrafilter $\mu'$ containing $B$ the cone $\Cone_{\mu'}(X,e,\omega)$ has non-trivial fundamental group. The cones are therefore non-homeomorphic.\\[1ex]
Note that together with the proof of Proposition \ref{scaling} we obtain the following. Let $\alpha$ be the sequence of scaling factors obtained by ordering the elements of $A$ in the natural order and $\beta$ the sequence obtained from the set $B$. Then the asymptotic cone $\Cone_\mu(X,e,\alpha)$ is simply connected for {\bf any} ultrafilter $\mu$, because it is by Proposition \ref{scaling} isometric to $\Cone_{\mu'}(X,e,\omega)$ and $\mu'$ is an ultrafilter containing $A$ by construction.\\[1ex]
The same argument shows that $\Cone_\mu(X,e,\beta)$ has non-trivial fundamental group again independent of the choice of the ultrafilter $\mu$.\\[1ex]
This example shows that it is not enough to simply fix the scaling factor and vary the ultrafilter to get all possible asymptotic cones.
\end{example}

\section{Proper spaces as asymptotic cones}

We will now proceed to show that many different spaces can arise as asymptotic cones. In particular the following theorem holds.

\begin{thm}
\label{decone}
Let $(Y,d)$ be a proper metric space. Then there is a proper metric space $(X,\bar{d})$ with basepoint $e \in X$ and a sequence of scaling factors $(\alpha_n)$, such that for any non-principal ultrafilter $\mu$ on $\IN$ there is an isometry
\[ \Cone_\mu(X,e,\alpha) \cong Y.\]
\end{thm}

\begin{proof}
Set $\alpha_n := n!$ and fix any non-principal ultrafilter $\mu$ on $\IN$. Choose any point $e \in Y$ as basepoint. For $n \geq 2$ consider the following subset of $Y$:
\[ Y_n := \left\{ y \in Y : d(y,e) \in \left[ \frac{1}{\log n}, \log n \right] \right\} \cup \{e\}.\]
This is a closed subset of $Y$ and therefore itself a complete metric space. Rescale the metric on $Y_n$ by $n!$ and call the resulting space $X_n$, i.e. for $x,x' \in X_n$ we have $\bar{d}(x,x') = n! \cdot d(x,x')$.\\[1ex]
Define the space $X$ now as the union of the spaces $X_n$ amalgamated along the common basepoint $e$. For $x \in X$ with $x \not= e$ write $x < X_n$ if $x \in X_k$ for some $k < n$ and similarly write $x > X_n$ if $x \in X_k$ for some $k > n$.\\[1ex]
Consider now the asymptotic cone of $X$ with respect to $\alpha$, $\mu$ and the basepoint $e$. Suppose $[x_n]$ is any point in this cone represented by a sequence in $X$. Then there are three cases:
\begin{itemize}
\item[]{\it Case 1:} We have $\mu$-almost surely $x_n < X_n$. Then
\[ \lim_\mu \frac{\bar{d}(x_n,e)}{\alpha_n} \leq \lim_\mu \frac{(n-1)! \log (n-1)}{n!} = \lim_\mu \frac{\log(n-1)}{n} = 0.\]
In this case $(x_n)$ is equivalent to the constant sequence given by the basepoint.
\item[]{\it Case 2:} We have $\mu$-almost surely $x_n \in X_n$.
\item[]{\it Case 3:} We have $\mu$-almost surely $x_n > X_n$. But then
\[ \lim_\mu \frac{\bar{d}(x_n,e)}{\alpha_n} \geq \lim_\mu \frac{(n+1)!}{\log(n+1) n!} = \lim_\mu \frac{n+1}{\log n} = \infty.\]
In this case the sequence does not give a point in the asymptotic cone, contradicting the assumption.
\end{itemize}
This shows that any point in the asymptotic cone which is different from the basepoint must fulfill the condition of case 2 above.\\[1ex]
Now let $y \in Y$ be an arbitrary point. For $n \in \IN$ define 
\[ \phi_n(y) := \left\{ \begin{array}{ll} y \qquad& \mbox{, if }\frac{1}{\log(n)} \leq d(y,e) \leq \log(n)\\ e \qquad&\mbox{ otherwise.}\end{array}\right. \]
Then $\phi_n(y) \in X_n$ and for all $y \in Y$ there is a natural number $N$, such that for all $n \geq N$ we have $\phi_n(y) = y$. Define now a map $\phi: Y \to \Cone_\mu(X,e,\alpha)$ by setting $\phi(y) := [\phi_n(y)]$. The basepoint $e$ of $Y$ is then mapped to the class of the constant sequence $[e]$. By construction this map is an isometric embedding, for if $y,y' \in Y$ are arbitrary points we have
\[ \bar{d}_\infty\big(\phi(y),\phi(y')\big) = \lim_\mu \frac{\bar{d}\big(\phi_n(y),\phi_n(y')\big)}{\alpha_n} = \lim_\mu \frac{n! \cdot d(y,y')}{n!} = d(y,y').\]
We now have to prove that $\phi$ is a surjection to get the required isometry. For this let $[x_n]$ be an arbitrary point of the cone represented by a sequence $(x_n)$ in $X$. We may assume that this sequence is not equivalent to the basepoint. By the above condition we know that $\mu$-almost surely we have $x_n \in X_n$. Regarding the points $x_n$ as points in $Y$ we get the inequality
\[ d(x_n,e) = \lim_\mu \frac{n! \cdot d(x_n,e)}{n!} = \lim_\mu \frac{\bar{d}(x_n,e)}{\alpha_n} < \infty \]
since the point is by assumption in the cone. It follows that the sequence $(x_n)$ is bounded in $Y$. Since $Y$ is proper and therefore complete there is a limit $y$ of this sequence with respect to $\mu$. And since
\[ \bar{d}_\infty\big([x_n], \phi(y)\big) = \lim_\mu \frac{\bar{d}(x_n,y)}{\alpha_n} = \lim_\mu \frac{n! \cdot d(x_n,y)}{n!} = \lim_\mu d(x_n,y) = 0\]
the point $\phi(y)$ is equivalent to $(x_n)$ and therefore $\phi$ is a surjection.\\[1ex]
It remains to show that $X$ is again a proper metric space. First observe that any of the spaces $X_n$ is compact, since it is a rescaled version of a closed subset of the closed ball with radius $\log n$ around $e$ in the proper space $Y$. Moreover the basepoint $e$ in each of the $X_n$ is isolated and has distance at least $\frac{n!}{\log n}$ from any other point in $X_n$. Since this grows with $n$ it is clear that any closed ball with fixed radius around any point in $X$ only meets finitely many of the $X_n$ and can therefore be seen as a finite union of compact sets, which is compact itself.
\end{proof}

\begin{cor}
\label{itdecone}
Let $(Y,d)$ be a proper metric space. Then for each number $k \in \IN$ there is a proper metric space $(X^{(k)},\bar{d})$ with basepoint $e \in X$ and a sequence of scaling factors $(\alpha_n)$, such that for any non-principal ultrafilter $\mu$ on $\IN$ there is an isometry
\[ \Cone^k_\mu(X^{(k)},e,\alpha) \cong Y.\]
\end{cor}

\begin{proof}
Since the metric space $(X,\bar{d})$ from Theorem \ref{decone} is again proper, the process can be iterated.
\end{proof}

\begin{rmrk}
Instead of proving Theorem \ref{decone} and Corollary \ref{itdecone} for fixed scaling factor and all ultrafilters, we could have stated that there is a non-principal $\mu$, such that the theorem is valid for the scaling factor $\omega$ by Proposition \ref{scaling}.
\end{rmrk}

\section{Slow ultrafilters}

\begin{defi}
Let $A = \{a_1 < a_2 < a_3 < ... \} \subseteq \IN$. We call $A$ thin if $\lim \frac{a_n}{a_{n+1}} = 0$ and we call $A$ fast if $\lim \frac{a_n}{n} = \infty$.
\end{defi}

\begin{rmrk}
It is easy to see that every thin set is fast. The converse is not true, the set $A = \{2^n : n \in \IN\}$ is an example of a fast set which is not thin.
\end{rmrk}

\begin{lemma}
The collection $\mc{S}$ of all cofinite sets together with complements of fast sets forms a filter.
\end{lemma}

\begin{proof}
Since subsets of fast sets are clearly fast it remains to show that the union of two fast sets is again fast, which is a simple calculation.
\end{proof}

Any ultrafilter extending the filter from the lemma will be called slow. Dru\c{t}u and Sapir asked in \cite{DS} if there are examples of groups having non-homeomorphic cones with respect to the standard scaling sequence $\omega = (1,2,3,4,\ldots)$ and two slow ultrafilters.\\[1ex]
The relevance of this question lies in the fact that almost all examples of groups (or metric spaces in general) having different asymptotic cones rely on sequences of scaling factors which are thin when seen as subsets of $\IN$. Equivalently this means that these cones are formed using the standard scaling sequence $\omega$ and ultrafilters containing thin sets.\\[1ex]
The next theorem shows that any construction for cones that can be realised with an ultrafilter containing a thin set can also be done with a slow ultrafilter.

\begin{thm}
\label{Drutuanswer}
Let $A$ be a thin set and $\mu$ an ultrafilter containing $A$. Then there is a slow ultrafilter $\mu'$, such that for every pointed metric space $(X,e)$ there is an isometry
\[ \Cone_\mu(X,e,\omega) \to \Cone_{\mu'}(X,e,\omega).\]
\end{thm}

\begin{proof}
Fix the thin set $A = \{ a_1 < a_2 < a_3 < \ldots \}$. For every $L > 1$ and $n \in \IN$ set
\[ X_{L,{a_n}} := \left[ \frac{1}{L} a_n, L a_n\right] \cap \IN \qquad \mbox{and for } I \subseteq A \mbox{ set} \qquad X_{L,I} := \bigcup_{a_n \in I} A_{L,{a_n}}\]
Since $A$ is thin these intervals will be disjoint for large $n$, so it is no loss of generality to assume that this is always a disjoint union by getting rid of finitely many parts. We will first show that the set $X_{L,I}$ for any $I \subseteq A$ with $I \in \mu$ is not fast, neither is its complement.\\[1ex]
First note that an infinite set $X \subseteq \IN$ is fast if and only if
\[ \lim_{x \to \infty \atop x \in X} \frac{|X \cap[1,x-1]|}{x} = 0. \qquad (*)\]
In the set $X_{L,I}$ consider a subsequence of elements of the form $L a_n$ for $a_n \in I$. Then
\[ \frac{|X_{L,I} \cap [1,La_n - 1]|}{La_n} \geq \frac{La_n - \frac{1}{L}a_n - 1}{La_n} = 1 - \frac{1}{L^2} - \frac{1}{La_n}.\]
Since $L > 1$ this will be bounded away from 0 for $n \to \infty$. Therefore this subsequence doesn't satisfy (*) and this implies that the set $X_{L,I}$ is not fast.\\[1ex]
For the complement $Y = \IN \backslash X_{L,I}$ note that $Y$ contains sets of the form
\[ \left] L a_{n-1}, \frac{1}{L} a_n \right[ \cap \IN. \]
It is no loss of generality to consider a subsequence in $Y$ of elements of the form $\frac{1}{L}a_n$. We see
\[ \frac{|Y \cap [1,\frac{1}{L}a_n - 1]|}{\frac{1}{L}a_n} \geq \frac{\frac{1}{L}a_n - La_{n-1} - 1}{\frac{1}{L}a_n} = 1 - L^2 \frac{a_{n-1}}{a_n} - \frac{L}{a_n}\]
Since $A$ is thin, the right hand side is bounded away from 0 as $n$ goes to infinity, so again (*) is not satisfied for the complement.\\[1ex]
Consider now the collection of sets $\{ X_{L,I} : I \in \mu, I \subseteq A\}$ for some fixed $L > 1$. Since $\mu$ is an ultrafilter this collection is closed under taking intersections of finitely many sets. And since all these sets are not fast and their complements are not fast as well, we can find a filter $\mc{F}_L$ containing $\mc{S}$ and this family.\\[1ex]
Now fix a sequence $L_k > 1$ of numbers tending to 1 (strictly monoton). Then for each $I \subseteq A, I \in \mu$ we have $X_{L_k,I} \subseteq X_{L_r,I}$ for $L_k < L_r$. This means that each generating set of $\mc{F}_{L_r}$ contains a generating set of $\mc{F}_{L_k}$ and because filters are closed under taking supersets it follows that $\mc{F}_{L_r} \subseteq \mc{F}_{L_k}$. This implies that we obtain an ascending sequence of filters
\[ \mc{S} \subseteq \mc{F}_{L_1} \subseteq \mc{F}_{L_2} \subseteq \mc{F}_{L_3} \subseteq \ldots \]
A direct application of Zorn's Lemma yields an ultrafilter $\mu'$ containing all these filters. In particular $\mu'$ is a slow ultrafilter since it contains $\mc{S}$. Define a map
\[ \phi: \Cone_\mu(X,e,\omega) \to \Cone_{\mu'}(X,e,\omega)\]
by setting $\phi\big([x_m]\big) := [y_m]$ where $y_m$ need only be defined for $m \in X_{L_1,A}$, say. Set $y_m := x_{a_n}$ if $m \in X_{L_1,a_n}$. By construction this map is well-defined: Consider another sequence $[x_m']$ which agrees with $[x_m]$ on a set $I$ of $\mu$-measure 1. Since $A \in \mu$ it is no loss of generality to assume $I \subseteq A$. The construction then implies that the image under $\phi$ of these sequences agrees on the set $X_{L_1,I} \in \mu'$.\\[1ex]
Moreover $\phi$ is a bi-Lipschitz homeomorphism with constant $L_1$. This is immediate since each $k \in X_{L_1,A}$ is in exactly one interval $X_{L_1,a_n}$ and we have
\[ \frac{d(x_{a_n},e)}{L_1 a_n} \leq \frac{d\big(\phi(x_k),e\big)}{k} \leq \frac{d(x_{a_n},e)}{\frac{1}{L_1}a_n}.\]
Consider now an arbitrary $L_k$. Since $L_k \leq L_1$ we know that $X_{L_k,A} \subseteq X_{L_1,A}$. Note that the actual definition of $\phi$ does not depend on the constant $L_1$, therefore the map $\phi$ can be defined for all $L_k$ in the same way, it is actually the same map. It follows that $\phi$ is indeed a bi-Lipschitz map with constant $L_k$ for all $k$. Since $L_k \to 1$, we find that $\phi$ is the desired isometry.
\end{proof}
This theorem shows that it is possible to ``thicken'' an ultrafilter containing a thin set in such a way that the same cone can be realised using a slow ultrafilter. Therefore if one has an example of a finitely generated group with two different asymptotic cones using different ultrafilters containing thin sets, one can modify the construction to obtain two slow ultrafilters yielding different cones. For instance this can be done in the example given by Thomas and Velickovic in \cite{TV}.

\end{document}